\documentclass[12pt, reqno]{amsart}
\usepackage{amsmath, amsthm, amsxtra, amscd, amsfonts, latexsym, amssymb, mathrsfs, graphicx, color,ulem, pstricks}
\usepackage{color}
\usepackage[bookmarksnumbered, colorlinks, plainpages]{hyperref}

\hypersetup{colorlinks=true,linkcolor=black, anchorcolor=green, citecolor=cyan, urlcolor=red, filecolor=magenta, pdftoolbar=true}
\usepackage{setspace}
\theoremstyle{plain}
\newtheorem{theorem}{Theorem}[section]
\newtheorem{lemma}[theorem]{Lemma}
\newtheorem{proposition}[theorem]{Proposition}
\newtheorem{corollary}[theorem]{Corollary}
\theoremstyle{definition}
\newtheorem{definition}[theorem]{Definition}
\newtheorem{example}[theorem]{Example}

\theoremstyle{remark}

\numberwithin{equation}{section}

\definecolor{darkgreen}{rgb}{.1,.5,0}

\theoremstyle{plain}

\theoremstyle{definition}

\newtheorem*{Sketch of proof}{Sketch of proof}

\numberwithin{equation}{section} 

\renewcommand{\thefootnote}{\fnsymbol{footnote}}

\begin{document}
	
\title{Classes of operators related to subnormal operators}
	
\author[R.E. Curto and T. Prasad]{ Ra\'ul E. Curto and Thankarajan Prasad}
	
\date{}

\maketitle

%
%
%
%

\centerline{\bf Abstract}

In this paper we attempt to lay the foundations for a theory encompassing some natural extensions of the class of subnormal operators, namely the $n$--subnormal operators and the sub-$n$--normal operators. \ We discuss inclusion relations among the above-mentioned classes and other related classes, e.g., $n$--quasinormal and quasi-$n$--normal operators. \ We show that sub-$n$--normality is stronger than $n$--subnormality, and produce a concrete example of a $3$--subnormal operator which is not sub-$2$--normal. \ In \cite{CU1}, R.E. Curto, S.H. Lee and J. Yoon proved that if an operator $T$ is subnormal, left-invertible, and such that $T^n$ is quasinormal for some $n \le 2$, then $T$ is quasinormal. \ In subsequent work, \cite{JS}, P.Pietrzycki and J. Stochel improved this result by removing the assumption of left invertibility. \ In this paper we consider suitable analogs of this result for the case of operators in the above-mentioned classes. \ In particular, we prove that the weight sequence of an $n$--quasinormal unilateral weighted shift must be periodic with period at most $n$. \ \vspace{-10pt}



\setcounter{page}{1}


\renewcommand{\thefootnote}{}

\tableofcontents


%
%
%
%

\section{Introduction}

The study of classes of non-normal operators on infinite dimensional complex Hilbert space is one of the chief interests in operator theory. \ The class of subnormal operators, introduced by P.R. Halmos \cite{Halmos2} and initially developed by Halmos and J. Bram \cite{Bram,Halmos1,Halmos3} is an interesting extension of the well-studied class of normal operators. \ A Hilbert space operator is subnormal if it has a normal extension. \ It can be observed that the theory of subnormal operators found in the literature is not easy, and oftentimes a highly nontrivial extension of normal operator  theory (\cite{Conway2,Halmos1,Halmos3}. \ A more general operator class, that of hyponormal operators includes both the normal and the subnormal operator classes. \ An extensive study on this class can been found in \cite{Putinar,xia}. \ Even though various extensions of hyponormal operators have been investigated by many authors, studies on subnormal operators and  hyponormal operators revolve very near other classes of operators, like binormal, quasinormal, $k$--hyponormal, etc. \ Spectral and structural problems related to these operators have received increased attention from operator theorists.

Halmos\cite{Halmos3} gave a characterization for subnormal operators in terms of the action of the operator on finite sets of vectors in its domain. \ J. Bram\cite{Bram} sharpened this result. \ The Bram-Halmos characterization for subnormal operators states that an operator $T\in B(\mathcal{H})$ is subnormal if and only if  $\Sigma^{k}_{i,j=0} \langle T^{j}x_{i}, T^{i}x_{j}\rangle  \geq 0$ for every finite set $x_{0}, x_{1}, \ldots ,x_{k}$ in $\mathcal{H}$. \ This is equivalent to
\begin{displaymath} \label{e1}
	\left( \begin{array}{ccccc}
		I  &  T^{*}  & \cdots     & T^{*k} \\
		T& T^{*}T  & \cdots &   T^{*}T  &  \\
		\vdots &  \vdots &\vdots & \vdots\\
		T^{k}& T^{*}T^{k}  &  \cdots  &  T^{*k}T^{k}&  \\
	\end{array} \right) \geq 0\eqno(1)
\end{displaymath}
for all $k\geq 1$ \cite{bridge}. \ If $k=1$, then it is evident that $T$ is hyponormal. \ If $T$ satisfies condition (\ref{e1}) for a fixed $k$, then $T$ is called $k$--hyponormal. \ To provide a bridge between subnormality and hyponormality and subnormality, a pioneering study of $k$--hyponormal operators was done by R.E. Curto in \cite{curto4,curto1,curto2,curto3}.

Let $\mathcal{H}$ and $\mathcal{K}$ be separable complex Hilbert spaces, and let $B(\mathcal{H},\mathcal{K})$ denote the algebra of all  bounded linear operators from $\mathcal{H}$ to $\mathcal{K}$ (We also write $B(\mathcal{H})=B(\mathcal{H},\mathcal{H})$).
\ Recall that an operator $T\in B(\mathcal{H})$ is said to be \textit{$n$--normal} if  $T^{*} T^{n} = T^{n}T^{*}$\cite{patel}. \ Alternatively, an operator $T$ is $n$--normal if and only if $T^{n}$ is normal. \  The class of $n$--normal operators has been studied extensively in \cite{patel,cho2,cho1,duggal,Put}; recently, B. Duggal\cite{duggal} proved that $n$--normal operators are subscalar and satisfy Weyl's Theorem. \

In parallel with the above-mentioned developments, the theory of subnormal operators and related classes of operators have had a remarkable impact in a number of problems in  operator theory and mathematical physics; see, for instance, \cite{HS,If,Sz}. \ In this paper, we focus attention on two larger classes: the $n$--subnormal operators and the sub-$n$--normal operators. \ We begin with some notation and preliminaries. \ First, we briefly recall two classical families of bounded linear operators on Hilbert space. \ As usual, we say that an operator $T$ is {\it normal} if $T^*T=TT^*$, {\it hyponormal} if $T^*T \ge TT^*$, {\it quasinormal} if $T$ commutes with $T^*T$, {\it subnormal} if $T$ is the restriction of a normal operator to an invariant subspace, and {\it quadratically hyponormal} if $p(T)$ is hyponormal for every quadratic polynomial $p$. 

We briefly pause to alert the reader that a very different notion of $n$--normality exists in the literature. \ Motivated by the pioneering work of C. Pearcy and N. Salinas on $n$--normality of operators (which they defined as the operators unitarily equivalent to an $n \times n$ operator matrix whose entries are commuting normal operators, cf. \cite{PS,Sal}, and also \cite{Pau}), in 2006 E. Ko, I.B. Jung and C. Pearcy \cite{JKP} introduced and studied the so-called sub-$n$--normal operators, defined as the restriction of an $n$--normal operator (as defined in \cite{PS,Sal}) to an invariant subspace; however, this notion has not been further developed in the literature. \ On the other hand, the notion of $n$--normality introduced and studied by S.A. Alzuraiqi and A.B. Patel in 2010 has recently taken center stage, and it is nowadays widely considered the appropriate version of $n$--normality in Hilbert space. \ In this paper we focus on this notion and the associated sub-$n$--normality, while at the same time introducing the new notion of $n$--subnormality.

Consider now the Hilbert space $\ell^2$ with its standard orthonormal basis $\{e_j\}_{j=0}^\infty$ (note that we begin indexing at zero). \ Given a bounded sequence of positive real numbers $\alpha \equiv \{\alpha_j\}_{j \ge 0}$, we define the {\it unilateral weighted shift} $W_\alpha$ acting on $\ell^2$ by $W_\alpha e_j := \alpha_j e_{j+1}$, and extend it to all of $\ell^2$ by linearity. \  It is well-known that $W_{\alpha}$ is never normal, quasinormal if and only if it is a scalar multiple of the (un-weighted) unilateral shift $U_+$, and hyponormal if and only if the sequence $\alpha$ is non-decreasing. \ 

On the other hand, recall that the Hardy space of the unit circle $\mathbb{T}$ is the closed subspace $H^2 \equiv H^2(\mathbb{T})$ of $L^2 \equiv L^2(\mathbb{T},\frac{d\theta}{2 \pi})$ spanned by the polynomials $\mathbb{C}[z]$. \ The above-mentioned unilateral shift $U_+$ is (canonically) unitarily equivalent to the multiplication operator $M_z \in B(L^2(\mathbb{T},\frac{d \theta}{2 \pi}))$ restricted to $H^2(\mathbb{T})$. \ 

As is customary, we let $M_n \equiv M_{n \times n}$ denote the algebra of $n \times n$ matrices over $\mathbb{C}$. \ We denote by $L^2 _{\mathbb{C}^n}$ (resp. $L^{\infty} _{M_n}$)  the Hilbert space of all $\mathbb{C}^n$--valued Lebesgue square integrable functions on the unit circle (resp. the Banach space of all $M_n$--valued essentially bounded functions on the unit circle). \ For a given $\Phi$ $\in$ $L^{\infty}_{M_n}$, the block Toeplitz operator with symbol $\Phi$ is defined as $T_\Phi f := P_n (\Phi f)$ \quad ($f$ $\in$ $H^2_{\mathbb{C}^n}$) which is the corresponding Hardy space. \ (Here $P_n$ is the orthogonal projection of $L^2_{\mathbb{C}^n}$ onto $H^2_ {\mathbb{C}^n}$.) \ If we take $H^2_{\mathbb{C}^n} = H^2 (\mathbb{T}) \oplus \cdots \oplus H^2 (\mathbb{T})$, then it is easy to see that 

$$T_{\phi} = \begin{bmatrix}
	T_{\phi_{11}} & \cdots & T_{\phi_{1n}} \\
	& \vdots &  \\
	T_{\phi_{n1}} & \cdots & T_{\phi_{nn}}
\end{bmatrix},  \text{whenever}\,\, 
\Phi = \begin{bmatrix}
	\phi_{11} & \cdots & \phi_{1n} \\
	& \vdots &  \\
	\phi_{n1} & \cdots & \phi_{nn}
\end{bmatrix} .\\$$

(In an entirely similar way, we can define vectorial Toeplitz operators where the space of matrices $M_n$ is replaced by the algebra of bounded operators acting on a Hilbert space.)

\medskip
\section{Some Preliminary Results}

For easy reference, we first list a number of well-known results. \ Recall that an operator $T$ is said to be an {\it isometry} if $T^*T = I$, where $I$ denotes the identity operator.

\begin{lemma}
(i) \ For $n \ge 1$, an isometry is subnormal but not necessarily $n$--normal. \ 

(ii) \ On the vector-valued Hardy space on the unit circle, denoted by $H^2(\mathbb{T}) \otimes \mathbb{C}^n$, let $\Phi=\begin{bmatrix} 0& 1\\
      0& 0\\
   \end{bmatrix} \in L^{\infty}_{M_{2}}(\mathbb{T})$, regarded as the symbol of the vector-valued Toeplitz operator $T_\Phi$. \ Then $T_{\Phi}$  is not subnormal; actually, since $\Phi$ is not normal,  $T_{\Phi}$ cannot even be hyponormal, by a result of C.Gu, J. Hendricks and D. Rutherford \cite[Theorem 3.3]{GHR}. \ However, $T_{\Phi}$ is $2$-normal.

(iii) \ Combining (i) and (ii) above, we easily see that there is no inclusion relation between the class of subnormal operators and the class of $2$--normal operators.
\end{lemma}
 
\begin{lemma} (\cite[Example 2.3]{patel})
On the Hilbert space $\ell^2(\mathbb{Z}_+)$ of square summable sequences of complex numbers, with canonical orthonormal basis $\{e_n\}_{n ge 0}$, consider the operator $T$ which maps $e_0$ to itself, $e_{2k-1}$ to $e_{2k}$ \; (for $k \ge 1$), and $e_{2k}$ to $0$ \; (for $k \ge 1$). \ Then $T^2$ is the orthogonal projection onto the one-dimensional subspace spanned by $e_0$ (and therefore normal and compact), while $T$ is not compact, and neither hyponormal nor co-hyponormal. 
\end{lemma}

\begin{lemma} \ (i) (\cite[Example 2.4]{patel}) \ On $\ell^2(\mathbb{Z}_+)$, the (unweighted) unilateral shift $U_+$ is subnormal, but not $n$--normal for any $n \ge 1$. \newline
(ii) (\cite[Theorem 2.5]{patel}) \ The set of $n$--normal operators is closed in the norm topology, and closed under scalar multiplication. \newline
(iii) (\cite[Proposition 2.6]{patel}) \ If $T$ is $n$--normal, so is $T^*$. \newline 
(iv) (\cite[Proposition 2.6]{patel}) \ If $T$ is $n$--normal and invertible, so is $T^{-1}$. \newline 
(v) (\cite[Proposition 2.6]{patel}) \ If $T$ is $n$--normal and $S$ is unitarily equivalent to $T$, then $S$ is $n$--normal. \newline 
(vi) (\cite[Proposition 2.6]{patel}) \ If $T$ is $n$--normal and $\mathcal{M}$ is a reducing subspace for $T$, then $T|_{\mathcal{M}}$ is $n$--normal. \newline 
(vii) (\cite[Theorem 2.8]{patel}) \ If $S$ and $T$ are $n$--normal, and $ST=TS$, then $ST$ is $n$--normal. \ This result is not true if $S$ and $T$ do not commute. \newline 
(viii) (\cite[Corollary 2.10]{patel}) \ If $T$ is $n$--normal and $m \ge 1$, then $T^m$ is $n$--normal. \newline 
(ix) (\cite[Lemma 2.13]{patel}) \ If $S$ and $T$ are $2$--normal, and $ST+TS=0$, then $S+T$ and $ST$ are $2$--normal. \newline 
(x) (\cite[Proposition 2.19]{patel}) \ $T-\lambda$ is $n$--normal for all $\lambda \in \mathcal{C}$, then $T$ is normal. \ On the other hand, $I+T$ may fail to be $2$--normal.\newline 
(xi) (\cite[Proposition 2.20]{patel}) \ If $T \equiv A+iB$ ($A,B$ self-adjoint), then $T$ is $2$--normal if and only if $B^2$ commutes with $A$ and $A^2$ commutes with $B$. \newline 
(xii) (\cite[Examples 2.21 and 2.22]{patel}) \ A $2$--normal operator may fail to be $3$--normal; similarly, a $3$--normal operator may fail to be $2$--normal. \newline
(xiii (\cite[Proposition 2.23]{patel}) \ If $T$ is both $n$--normal and $(n+1)$--normal, then $T$ is $(n+2)$--normal. \newline 
(xiv) (\cite[Corollary 2.27]{patel}) \ If $T$ is both a partial isometry and a $2$--normal operator, then $T$ is $n$--normal for all $n \ge 3$. \newline 
(xv) (\cite[Proposition 2.33]{patel}) \ Let $T$ be an operator, let $n \ge 1$ and let $F:=T^n+T^*$ and $G:=T^n-T^*$. \ Then $T$ is $n$--normal if and only if $G$ commutes with $F$. \newline 
(xvi) (\cite[Proposition 2.39]{patel}) \ If $T$ is $n$--normal and quasinilpotent, then $T$ is nilpotent.
\end{lemma}

\section{The classes of $n$--subnormal operators and sub-$n$–normal operators}

We now define the classes of $n$--subnormal operators and sub-$n$--normal operators as natural extensions of the classes of normal and subnormal operators. \ Hereafter, $n$ will denote a fixed positive integer.

\begin{definition}
An operator $T\in B(\mathcal{H})$ is said to be $n$--subnormal if $T^{n}$ is subnormal.
\end{definition}

Our interest in the class of $n$--subnormal operators is partly motivated by a long-standing open question in operator theory, recorded as Problem 5.6 in \cite{ConwayFeldman}: Characterize the subnormal operators having a square root. \ This question has been recently considered by J. Mashreghi, M. Ptak and W. Ross in \cite{MPR}.

\begin{definition}
An operator $T\in B(\mathcal{H})$ is said to be sub-$n$--normal if it is the restriction of an $n$--normal operator to an invariant subspace; that is, there exists a Hilbert space $\mathcal{K}$ containing $\mathcal{H}$ and an $n$--normal operator $S$ on $\mathcal{K}$ such that $S\mathcal{H} \subseteq \mathcal{H}$ and $T = S |_{\mathcal{H}}$.

\end{definition}
It is easy to see that sub-$n$--normal operators generally admit many non-unitarily equivalent $n$--normal extensions. \ In the sequel (see Theorem \ref{thm 229}), we will identify a unique (up to unitary equivalence) {\it minimal} $n$--normal extension, just as it happens with subnormality.

It is well known that every subnormal operator is hyponormal. \ Now we will see that an $n$--subnormal operator need not be hyponormal when $n > 1$. 

   \begin{example}
 Consider the matrix $A = 
\begin{bmatrix}
0 & 1\\
0 & 0\\
\end{bmatrix}$. \ It is easy to see that $A^2 = 0$ and therefore, $A$ is $2$--subnormal. \ Moreover, $A$ is not hyponormal, as a straightforward calculation reveals.

\end{example}

Next, we will show that a hyponormal operator need not be $2$--subnormal. \ 

\begin{example} 
 Let $T = 2U + U^{\ast}$. \ It is well known that $T$ is hyponormal but $T^2$ is not hyponormal. \ It follows that $T$ is not $2$--subnormal. 
\end{example}

\begin{example} \label{ex26}
Let $M= 
\begin{bmatrix}
a & b\\
c & -a\\
\end{bmatrix} \in M_2(\mathbb{C})$, where $|b| \neq |c|$. 
Clearly, $M$ is not normal. \ But, $M^2 = 
\begin{bmatrix}
a & b\\
c & -a\\
\end{bmatrix}
\begin{bmatrix}
a & b\\
c & -a\\
\end{bmatrix} = 
\begin{bmatrix}
a^2+bc & 0\\
0 & a^2+bc\\
\end{bmatrix}$
is normal. \ This shows that  $M$ is $2$--normal (and a fortiori $T$ is sub-2-normal), but $T$ is not even hyponormal, much less subnormal.
\end{example} 

\begin{example} \label{ex27}
For $P$ a nontrivial projection, the $2 \times 2$ operator matrix  
$\begin{bmatrix} 
    I& P\\
    0& -I\\
   \end{bmatrix}$  
	is $2$-normal but not hyponormal.
   \end{example}
	
Related to Examples \ref{ex26} and \ref{ex27}, we briefly pause to state a fundamental result about $2$--normality, proved by H. Radjavi and P. Rosenthal in 1971. \ First, we recall that, given a bounded operator $T$, we let $\sigma(T)$, $\sigma_{a}(T)$, and $\partial \sigma(T)$ denote the spectrum, the approximate point spectrum, and the boundary of the spectrum, respectively.
	
\begin{theorem}	(\cite[Theorem 1]{RadRos}) \ An operator is the square root of a normal operator if and only if it is of the form
$$
A \oplus\left(\begin{array}{rr}
B & C \\
0 & -B
\end{array}\right),
$$
where $A$ and $B$ are normal, and $C$ is a positive one-to-one operator commuting with $B$. \ Furthermore, $B$ can be chosen so that $\sigma(B)$ lies in the closed upper half-plane and the Hermitian part of $B$ is positive. \ (Of course either direct summand may be absent in the above expression.)
\end{theorem}

\begin{proposition} \ Let $T$ be a hyponormal operator on Hilbert space. \ If $T$ is $2$-normal, then $T$ must be normal. \end{proposition}
   \begin{proof}
By the previous theorem, if $T$ is $2$--normal then $T$ is of the form 
$$
T=A \oplus \begin{bmatrix} 
      B& C\\
     0& -B\\
 \end{bmatrix},
$$
where $A$ and $B$ are normal and $C$ is a positive one-to-one operator commuting with $B$. \ It follows that $T$ is hyponormal if and only if the $2 \times 2$ operator matrix in the above expression, denoted by $Z$, is hyponormal. \ A straightforward calculation reveals that the $(1,1)$-entry of the self-commutator of $Z$ is $B^*B-BB^*-CC^* = -CC^*$, and this forces $C=0$, in which case $T$ is normal.
\end{proof}
  
\begin{example} 
Given a positive integer $k$, let $\mathcal{N}_k^{(n)}$ denote the set of $n$--normal functions, i.e., those functions $\Phi:\mathbb{T} \rightarrow M_k$ such that $\Phi(z)^n$ is a normal matrix $k \times k$ matrix a.e. on $\mathbb{T}$. \ If $\Phi \in L^{\infty}_{M_{k}} \cap \mathcal{N}_k^{(n)}$, then $M_{\Phi}$ acting on $L^{2}_{\mathbb{C}^{k}} $  is $n$--normal. \ It follows that, for $\Phi \in H^{\infty}_{M_{k}} \cap \mathcal{N}_k^{(n)}$, we have that $M_{\Phi}$ restricted to the invariant subspace $H^{2}_{\mathbb{C}^{k}} $ is sub-$n$--normal.
\end{example} 
	
\begin{example}
	Let $ \Phi=	\begin{pmatrix}
		0 & z  \\
		0 & 0
\end{pmatrix} \in L^{\infty}(M_{2}$). \ Since  $ \Phi$ is not normal, $T_{\Phi}$ is not hyponormal (again, by \cite[Theorem 3.3]{GHR})  and so $T_{\Phi}$ is not subnormal. \ But $M^{2}_{\Phi}=0$, and so $M^{2}_{\Phi}=0$ is normal. \ Therefore, $T_{\Phi}$ is sub-$2$-normal. \ On the other hand, since $T^{2}_{\Phi} =0$, we see that $T_{\Phi}$ is $2$--subnormal.	\ Therefore, $T_{\Phi}$ is both sub-$2$--normal and $2$--subnormal.
\end{example}

\begin{example}\label{ee}
Let $T$ be a weighted shift with weights $\{a, b, 1,1,1, \ldots \}$ where $0<a<b<1$. \ The operator $T$ is hyponormal but not subnormal by \cite[Problem 160]{Halmos3}. \ However, $T^2$ is unitarily equivalent to the direct sum of two subnormal weighted shifts, with weight sequences $ab,1,1,1,\ldots$ and $b,1,1,1,\ldots$, respectively. \ It follows that $T$ is $2$--subnormal but not even quadratically hyponormal (using \cite[Theorem 2]{CurtoQHWS}). \ (These shifts were studied in detail in \cite{curto}.)
\end{example}

\begin{example} \ For $0<a<b<c<1$, consider the unilateral weighted shift $W_{\alpha}$ with weight sequence $\alpha_0:=a,\alpha_1:=b,\alpha_2:=c,\alpha_3:=1,\alpha_4:=1,\ldots$\;. \ It is well known that $W_\alpha$ is hyponormal. \ However, $W_\alpha^2$ is not subnormal, being unitarily equivalent to the orthogonal direct sum of two weighted shifts, with weight sequences $ab, c, 1, \ldots$ and $bc,1,1,\ldots$. \ Moreover, $W_\alpha$ is not quadratically hyponormal \cite{CurtoQHWS,bridge}, and this implies that $W_\alpha$ cannot be subnormal. \ On the other hand, $W_\alpha$ is $3$--subnormal, since $W_\alpha^3$ is unitarily equivalent to the orthogonal direct sum of three weighted shifts, with weight sequences $abc, 1, 1, \ldots$, $bc,1,1,\ldots$, and $c,1,1,\ldots$\;. \ 

\noindent We claim that $W_\alpha$ is not sub-$2$--normal. \ Assume, to the contrary, that $W_\alpha$ is the restriction of a $2$--normal operator, that is, there exists a Hilbert space $\mathcal{K} \supseteq \mathcal{H}$ and a $2$--normal operator $T \in B(\mathcal{K})$ such that
$$
T=\left( 
\begin{array}{cc}
W_\alpha & R \\
0 & S
\end{array}
\right),
$$
with respect to the orthogonal decomposition $\mathcal{K} = \mathcal{H} \oplus \mathcal{H}^\perp$. \ Assume further that $T$ is the minimal $2$-normal extension of $W_\alpha$, that is, $T$ does not have any nontrivial reducing subspaces. \ 

We now recall the Radjavi-Rosenthal representation of $2$--normal operators \cite{RadRos}, that is, $T$ is unitarily equivalent to a direct sum of the form 
$$
A \oplus \left( 
\begin{array}{cc}
B & C \\
0 & -B
\end{array}
\right),
$$
where $A$ and $B$ are normal, and $C$ is a positive injective operator commuting with $B$. \ Since we are considering the minimal $2$-normal extension of $W_\alpha$, the normal part $A$ must be absent since otherwise the space on which $A$ acts would correspond (via the unitary equivalence) to a reducing subspace of $T$, contradicting minimality. \ Therefore $T$ is unitarily equivalent to 
$$
Z:=\left( 
\begin{array}{cc}
B & C \\
0 & -B
\end{array}
\right).
$$
We now compute the squares of $T$ and $Z$. \ We see that 
$$
T^2=\left( 
\begin{array}{cc}
W_\alpha^2 & A_\alpha R + RS \\
0 & S^2
\end{array}
\right),
$$
while
$$
Z^2 = \left( 
\begin{array}{cc}
B^2& 0 \\
0 & B^2
\end{array}
\right).
$$
Since $Z^2$ is normal, so must be $T^2$, and since $\mathcal{H}$ is invariant under $T$, we conclude that $W_\alpha^2$ is subnormal, a contradiction. \ We have therefore established that $3$--subnormal operators may fail to be sub-$2$-normal.  
\end{example}
	
The following result is an easy extension of Proposition 19.1.7 in \cite{Book1}; we give a proof, for the reader's convenience.
\begin{proposition}
	An operator that is unitarily equivalent to a sub-$n$--normal operator is sub-$n$--normal.
\end{proposition}

\begin{proof}
Let $T \in B(\mathcal{H}_{1}) $ be sub-n-normal and let $S \in B(\mathcal{K}_{1})$ be n-normal extension of $T$.  Suppose that $U: \mathcal{H}_{1} \rightarrow \mathcal{H}_{2}$ is a unitary operator such that $UTU^*=A \in B(\mathcal{H}_{2})$.  Let  $V: \mathcal{K}_{1} \rightarrow \mathcal{K}_{2}$ be the operator defined by $V=U \oplus I$,  where $\mathcal{K}_{1} =  \mathcal{H}_{1} \oplus (\mathcal{K}_{1} \ominus \mathcal{H}_{1}) $ and $\mathcal{K}_{2}=\mathcal{K}_{2} =  \mathcal{H}_{2} \oplus (\mathcal{K}_{1} \ominus \mathcal{H}_{1}) $.  Since $V$ is unitary and $S$ is n-normal, we have
$$(VSV^{*})^{n}(VS^{*}V^{*})^{n}=VS^{n} S^{*n}V^{*}=VS^{*n} S^{n}V^{*}= (VS^{*}V^{*})^{n}(VSV^{*})^{n}$$
That is ,  $VSV^{*}$ is n-normal operator on $\mathcal{K}_{2}$ and if $x\in \mathcal{H}_{2}$,  then it is easy to see that $VSV^{*} \mid_{\mathcal{H}_{2}}  =A$.  Consequently,  $A$ has a $n$--normal extension. 
\end{proof}

Appealing to \cite[Theorem 1]{Bram} (and its proof), we now state:
\begin{proposition}\label{Br}\cite{Bram}
	An operator $T\in B(\mathcal{H})$ is $n$--subnormal if and only if  \\ $\Sigma^{k}_{i,j} \langle (T^n)^{j}x_{i}, (T^{n})^{i}x_{j}\rangle \geq 0$ for every finite set $x_{0}, x_{1}, \ldots ,x_{k}$ in $\mathcal{H}$.
\end{proposition}

\begin{theorem}
	If $T$ is sub-$n$--normal, then it is the $n$--th root of a hyponormal operator; equivalently, $T^n$ is hyponormal.
\end{theorem}
\begin{proof}
	Suppose $S$  be the $n$--normal extension on $\mathcal{K}$  of $T \in B(\mathcal{H})$. \  We have $Ty=Sy$ for all $y\in \mathcal{H}$. \ Let $P$ be the projection from $\mathcal{K}$ on to $\mathcal{H}$.
	Now, $$\langle T^{*n}x, y\rangle=\langle x, T^{n}y\rangle=\langle S^{*n}x, P y\rangle=\langle PS^{*n}x,  y\rangle$$
	Since the operator $PS^{n}$ on $\mathcal{K}$ leaves $\mathcal{H}$ invariant, its restriction to $\mathcal{H}$ is
	an operator on $\mathcal{H}$. \ From the above calculations, $T^{*n}x= PS^{*n}x$ $x\in \mathcal{H}$. \ Then\\
	$\| T^{*n}x\|= \| PS^{*n}x\| \leq \| S^{*n}x\|=\| S^{n}x\|$ (by the normality of $S^{n}$ )= $\| T^{n}x\|$.
	That is,  $T^{n}$ is hyponormal.
\end{proof}
In \cite{Halmos3}, Halmos characterized subnormal operators as follows.
\begin{theorem}(\cite{Halmos3})
	An operator $T\in B(\mathcal{H})$ is subnormal if and only if (1) $\Sigma^{k}_{m,n} \langle T^{n}x_{m}, T^{m}x_{n}\rangle \geq 0$ for every finite set $x_{0}, x_{1}, \ldots ,x_{k}$ in $\mathcal{H}$., and (2) for every finite set $x_{0}, x_{1}, \ldots ,x_{k}$ in $\mathcal{H}$,  there exist a positive constant $c$ such that $\Sigma^{k}_{m,n} \langle T^{n+1}x_{m}, T^{m+1}x_{n}\rangle \leq c \cdot \Sigma^{k}_{m,n} \langle T^{n}x_{m}, T^{m}x_{n}\rangle$
\end{theorem}

We now state and prove, for sub-$n$--normal operators, a theorem that mimics the above result. 

\begin{theorem}\label{t1}
	If $T\in B(\mathcal{H})$ be a sub-$n$--normal operator, then \\(1) $\Sigma^{k}_{i,j=0} \langle (T^{n})^{j}x_{i}, (T^{n})^{i}x_{j}\rangle \geq 0$ for every finite set $x_{0}, x_{1}, \cdot ,x_{k}$ in $\mathcal{H}$, and \\(2) there exist a positive constant $c$ such that
	$$\Sigma^{k}_{i,j=0} \langle (T^{n})^{j+1}x_{i}, (T^{n})^{i+1}x_{j}\rangle \leq c \Sigma^{n}_{i,j=0} \langle (T^{n})^{j}x_{i}, (T^{n})^{i}x_{j}\rangle$$
	for every finite set $x_{0}, x_{1}, \ldots ,x_{k}$ in $\mathcal{H}$.
\end{theorem}

\begin{proof}
	By the definition of sub-$n$--normal operator, there exists an $n$--normal operator $S$ on $\mathcal{K}$ such that $\mathcal{H} \subseteq \mathcal{K}$ and $Tx=Sx$ for all $x\in \mathcal{H}$, and so $T^{*n}x= PS^{*n}x$.
	Let $x_{0}, x_{1}, \ldots ,x_{k}$ in $\mathcal{H}$. \ Then
	\begin{align*}
		\Sigma^{k}_{i,j=0} \langle (T^{n})^{j}x_{i}, (T^{n})^{i}x_{j}\rangle &=\Sigma^{k}_{i,j=0} \langle(S^{n})^{j}x_{i}, (S^{n})^{i}x_{j}\rangle
		\\&=\Sigma^{k}_{i,j=0} \langle (S^{n})^{* i}(S^{n})^{j}x_{i}, x_{j}\rangle
		\\&=\Sigma^{k}_{i,j=0} \langle (S^{n})^{j} (S^{n})^{*i}x_{i}, x_{j}\rangle   \quad (\textrm{since } \text{ } (S^{n})^*(S^{n})=(S^{n})(S^{n})^* )
		\\&=\Sigma^{k}_{i,j=0} \langle (S^{n})^{*i}x_{i}, (S^{n})^{*j}x_{j}\rangle=\| \Sigma^{k}_{j} (S^{n})^{*j}x_{j}\|^{2}.
	\end{align*}
	This completes the first part of the proof. \ Putting $y_{i}:=T^{n}x_{i}$ in the preceding expressions, it is easy to see that
	$$\Sigma^{k}_{i,j=0} \langle (T^{n})^{j+1}x_{i}, (T^{n})^{i+1}x_{j}\rangle \leq  c \Sigma^{k}_{i,j=0} \langle (T^{n})^{j}x_{i}, (T^{n})^{i}x_{j}\rangle$$
	for every finite set $x_{0}, x_{1}, \ldots ,x_{k}$ in $\mathcal{H}$, which establishes the second  part of the proof.
\end{proof}

\begin{corollary}
If $T$ is sub-$n$--normal, then $T$ is $n$--subnormal.
\end{corollary}

\begin{proof} This follows easily from Theorem \ref{t1}. \ Alternatively, let $T$ be sub-$n$--normal. \ Then, by definition,  $T$ has an $n$--normal extension $S$ on $\mathcal{K}$. \ Given $y \in \mathcal{H}$, we have  $T^{n}y = S^{n}y$; that is, $S^n$ is a normal extension of $T^n$. \ It follows that $T$ is $n$--subnormal.
\end{proof}

Using the preceding results and examples, we obtain the following implications: \bigskip

$\begin{array}{ccccc} 
&& \framebox[2.0in][c]{
$\begin{array}{c}
\textrm{subnormal} \\
\end{array}$
} && \\
& \mbox{\Huge{\rotatebox{45}{$\Longleftarrow$}}} \hspace*{-.1in} && \mbox{\Huge{\rotatebox{90}{\rotatebox{45}{$\Longleftarrow$}}}} \\

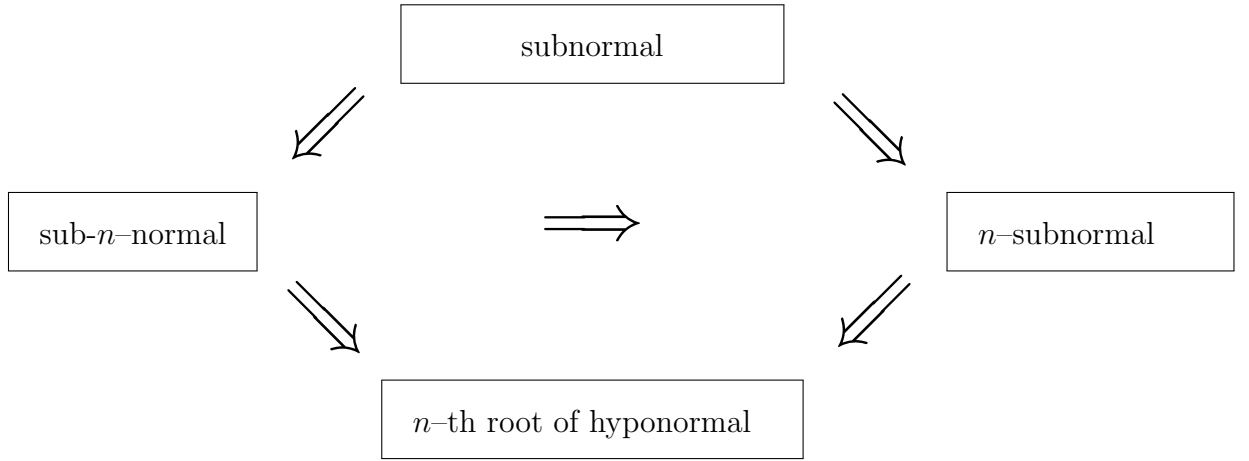
\framebox[1.3in][l]{
$\begin{array}{c}
\textrm{sub-$n$--normal} \\
\end{array}$ 
}  &&\mbox{\Huge{$\Longrightarrow$}}&& 
\framebox[1.5in][l]{
$\begin{array}{c}
\textrm{$n$--subnormal} \\
\end{array}$
} \\
&  \mbox{\Huge{\rotatebox{90}{\rotatebox{45}{$\Longleftarrow$}}}} \hspace*{-.1in} && \mbox{\Huge{\rotatebox{45}{$\Longleftarrow$}}} & \label{diagram}  \\
&& \framebox[2.2in][l]{
$\begin{array}{c}
\textrm{$n$--th root of hyponormal} \\
\end{array}$}
\end{array}$

\bigskip
\begin{definition}
	Let $T$ be a sub-$n$--normal on $\mathcal{H}$ and let $S$ be an $n$--normal extension of $T$ on $\mathcal{K}$. \ We shall say that $S$  is a minimal $n$--normal extension of $T$ if $\mathcal{K}$ has no proper subspace containing it to which the restriction of $S$ is also a $n$--normal extension of $T$.
\end{definition}

\begin{theorem} \label{thm 229}
	Let $T$ be a sub-$n$--normal on $\mathcal{H}$ and let $S$ be an $n$--normal extension of $T$ on $\mathcal{K}$. \ Then $S$ is minimal $n$--normal extension of $T$ if and only if 
	$$
	\mathcal{K}=
	\bigvee\{(S^{*})^{nk}h: h\in \mathcal{H} \text{ and
	} k=0,1\}.
	$$
\end{theorem}	

\begin{proof}
	Write $\mathcal{L}=\bigvee\{(S^{*})^{nk}h; h\in \mathcal{H}  \text{ and
	}k=0,1\}$. \ It is evident that $\mathcal{H}\subseteq \mathcal{L}$. \ Since $S$ is $n$--normal, $S^{*}$ is $n$--normal and so we see that  $S \mathcal{L}\subseteq \mathcal{L}$, by using the fact that $Sh=Th$ for $h\in \mathcal{H}$. \ 	
	Let $S_{\mathcal{L}}=S|_{\mathcal{L}}$. \ Since $S^{*n}f \in L$ and $S \mathcal{L}\subseteq \mathcal{L}$, we have $S_{\mathcal{L}}^{*n}S_{\mathcal{L}}f=P_{L}S^{*n}Sf=P_{\mathcal{L}}SS^{*n}f=SS^{*n}f$
	and 	$S_{\mathcal{L}}S_{\mathcal{L}}^{*n}f=SP_{\mathcal{L}}S^{*n}f=SS^{*n}f$ for $f\in \mathcal{H}$ and $P_{\mathcal{L}}$ the orthogonal projection of $\mathcal{K}$ onto $\mathcal{L}$. \ So $S_{\mathcal{L}}$ is $n$--normal and $\mathcal{H}\subseteq ker[S_{\mathcal{L}}^{*n}, S_{\mathcal{\mathcal{L}}}]$. \ Now we have to show that $\mathcal{L}$ is minimal. \ Let $\mathcal{M}\subseteq \mathcal{K}$ be such that $S\mathcal{M}\subseteq \mathcal{M}$ and $S|_{\mathcal{M}}$ is $n$--normal extension of $T$. \ Since $\mathcal{H}\subseteq ker[S_{\mathcal{M}}^{*n}, S_{\mathcal{M}}]$, $P_{\mathcal{M}}S^{*n}Sf= SP_{\mathcal{M}}S^{*n}f$ for $f\in \mathcal{H}$.\\
	Now,
	\begin{align*}
		\langle S^{*}f,  P_{\mathcal{M}}S^{*n}f\rangle&= 	\langle f,  SP_{\mathcal{M}}S^{*n}f\rangle\\
		&=\langle f,  P_{\mathcal{M}}S^{*n}Sf\rangle\\
		&=\langle f, S^{*n}Sf\rangle\\
		&=\langle f, SS^{*n}f\rangle\\
		&=	\langle S^{*}f, S^{*n}f\rangle
	\end{align*}
	Hence $S^{*n}f=P_{\mathcal{M}}S^{*n}f$ and so $S^{*n}f \in \mathcal{M}$, i.e., $\mathcal{L}\subseteq \mathcal{M}$. \ This completes the proof.
\end{proof}

\begin{proposition}
	For $k=1,2$, let $T_{k}$  be a sub-$n$--normal operator on $\mathcal{H}_{k}$ and let $S_{k}$ be a minimal $n$--normal extension on the Hilbert space $\mathcal{K}_{k}$. \ If $T_{1}$  and $T_{2}$ are unitarily equivalent then so are the minimal $n$--normal extensions $S_1$ and $S_2$. 
\end{proposition}
\begin{proof}
	Suppose that $U: \mathcal{H}_{1} \rightarrow \mathcal{H}_{2}$ is a unitary operator such that $UT_{1}=T_{2}U$. \ Define $V$ on $\mathcal{K}_{1}$ by
	$$
	V(S^{*nk}_{1}h)=S^{*nk}_{2}Uh, ( h\in \mathcal{H}_{1}, k=0,1).
	$$
Note that $V\mid_{H_{1}}=U$. \ For $f,g \in \mathcal{H}_{1}$,
		\begin{align*}
			||Uf+S^{*n}_{2}Ug||& =( Uf+S^{*n}_{2}g, Uf+S^{*n}_{2}g)\\
			&=(Uf,Uf)+(Ug, S^{n}_{2}Uf)+ (S^{n}_{2}Uf, Ug)+ (S^{n}_{2}Ug, S^{n}_{2}Ug) \\ &(\textrm{because } S^{n}_{2}S^{*n}_{2}=S^{*n}_{2}S^{n}_{2} \textrm{ for all } h \in \mathcal{H}_{2}) \\
			&=(Uf,Uf)+(Ug, US^{n}_{1}f)+ (US^{n}_{1}f, Ug)+ (US^{n}_{1}g, US^{n}_{1}g) =||Uf+S^{*n}_{1}g || \\
			&(\textrm{because U is unitary  and } UT_{1}=T_{2}U \Rightarrow US_{1}=S_{2}U \Rightarrow US^{n}_{1}=S^{n}_{1}U )
		\end{align*}
		
		Thus $$
		V (f+S^{*n}_{1})=Uf+S^{*n}_{2}Ug.
		$$
		$V$ extends to a unitary operator from $\mathcal{K}_{1}$ onto $\mathcal{K}_{2}$, Moreover,  $VS_{1}=S_{2}V $ holds from the following observation. \ For $h\in \mathcal{H}_{1}$
		$$ VS_{1}h=US_{1}h=S_{2}Uh=S_{2}Vh$$
		$$ VS_{1}(S^{*n}_{1}h)=VS^{*n}_{1}S_{1}h=S^{*n}_{2}US_{1}h=S^{*n}_{2}S_{2}Uh=S_{2}S^{*n}_{2}Uh=S_{2}V(S^{*n}_{1}h)$$
		This completes the proof.
	\end{proof}
	
	\begin{corollary}
		If $T\in \mathcal{B}(\mathcal{H})$ is sub-$n$--normal operator and $S_{1}$  and $S_{2}$ are  minimal $n$--normal extensions of $T$, then  $S_{1}$  and $S_{2}$ are  unitary equivalent.
	\end{corollary}

The following result extends to sub-$n$--normality a well-known fact in the theory of subnormal operators (cf. \cite[Problem 200]{Halmos3}.

\begin{proposition}
If $T$ is sub-$n$--normal and $S$ is the minimal $n$--normal extension of $T$, then  $\sigma_{a}(T) \subseteq \sigma(S) \subseteq \sigma(T)$ and $\partial \sigma(T) \subseteq \partial \sigma(S)$.
\end{proposition}

\begin{proof}
If $\lambda \in \sigma_{a}(T)$, then there is a sequence of unit vectors $\{g_{n}\}$ in $\mathcal{H}$  such that $ || (T-\lambda)g_{n}|| \rightarrow 0$ and so $ || (S-\lambda)g_{n}|| \rightarrow 0$. \ Thus, $  \sigma_{a}(T) \subseteq \sigma_{a}(S) \subseteq \sigma(S) $. \ Now we show that $\sigma(S) \subseteq \sigma(T)$. \ It is enough to prove if $T$ is invertible then so is $S$. \ Since $S$ is $n$--normal, $S^{n}$ is normal. \ Then, by the spectral theorem, $S^{n}=\int zdE(z)$, and let $\mathcal{M}=E(B(0,\varepsilon))\mathcal{K}$ be the reducing subspace for $S^{n}$. \ Then for $f \in \mathcal{M}$, $ || (S^{n})^{k}f||\leq \varepsilon^{k}f$, $k=1,2,3.....$. \ Now for $f \in \mathcal{M}$ and $g \in \mathcal{M}$
\begin{align*}
|(f,g)| &= | (f,T^{nk}T^{-nk}g)|\\
&=|(f,S^{nk}T^{-nk}g)|\\
&=|(S^{*nk} f,  T^{-nk}g)|\\
&\leq || (S^{*nk} f||\cdot ||T^{-nk}||\cdot ||g||
\end{align*}
Since $S^{n}$ is normal, $|(f,g)| \leq  \varepsilon^{k} ||f||\cdot ||T^{-nk}||\cdot ||g||$. \ If $\varepsilon < ||T^{-1}||^{-n}$,  then  $\varepsilon ||T^{-1}||^{n}<1$. \ Thus  $(\varepsilon ||T^{-1}|| ^{n})^{k} \rightarrow 0 $ as  $k \rightarrow \infty$. \ Hence $(f,g)=0$. \ That is $\mathcal{M}\perp \mathcal{H}$  if $\varepsilon < ||T^{-n}||^{-1}$ and so $\mathcal{H} \subseteq \mathcal{M}^{\perp}$. \ But   $S^{n}\mid_{\mathcal{M}^{\perp}}$ is a normal extension and so $S \mid_{\mathcal{M}^{\perp}}$ 
is an $n$--normal extension of $T$. \ By the minimality of $n$--normal extension, $\mathcal{M}^{\perp}=\mathcal{K}$ and so $\mathcal{M}=\{0\}$. \ Thus  $S^{n} $ is invertible and so $S$ is invertible. \  Since $\sigma_{a}(T) \subseteq \sigma(S) \subseteq \sigma(T)$ , it is easy to see that $\partial \sigma(T) \subseteq \partial \sigma(S)$.
\end{proof}

In the following result, by a {\it spectral set} for a bounded operator $T$ we mean, as is customary, a set $X \subseteq \mathbb{C}$ such that $\sigma(T) \subseteq X$ and von Neumann's inequality holds for $T$ on $X$.

\begin{corollary}
For a sub-$n$--normal operator $T$,  $\sigma(T)$ is a spectral set. 
\end{corollary}
\begin{proof}
Let $f$ be a rational functions with no poles in $\sigma(T)$. \ Then By above theorem, it follows that 
$ |f(T)| \leq sup\{ f(\mu) : \mu \in  \sigma(T) \}$. 
\end{proof} 
For subnormal operators, the following result was obtained by J. Bram in \cite{Bram}.  
\begin{proposition}
Let $T$ be a sub-$n$--normal and let $S$ be the minimal $n$--normal extension of $T$. \ If $Y$ is a bounded connected component of the complement of $\sigma(S)$  in $\mathbb{C}$, then $Y$ and  $\sigma(T)$ are disjoint or $Y \subseteq \sigma(T)$.
\end{proposition}
\begin{proposition}
Let $T$ be a sub-$n$--normal and let $S$ be the minimal $n$--normal extension of $T$. \ If $Y$ is a bounded connected component of the complement of $\sigma(S)$  in $\mathbb{C}$. \ Then the following statements are equivalent. \\
(a)  $Y \cap  \sigma(T)=\emptyset$. \\
(b) For each $\lambda \in Y$, $(S-\lambda)\mathcal{H}=\mathcal{H}$.\\
(c) For each $\lambda \in Y$, $(S-\lambda)^{-1}\mathcal{H} \subseteq \mathcal{H}$.
\end{proposition}

Let $C^{m}_{0} (\mathbb{C})$ is the space of compactly supported functions on $\mathbb{C}$, continuously differentiable of order $ m$, where $0 \leq m \leq\infty$.
An operator $T \in B(H)$  is said to be  scalar of order $m$  if if there is a continuous unital morphism of topological algebras
$$\Phi: C^{m}_{0} (\mathbb{C})\rightarrow B(\mathcal{H})$$ such that $\Phi(z) = T$, where $z$ is the identity function on $\mathbb{C}$.

\begin{theorem}
	Let  $T$ be a sub-$n$--normal operator with rich spectrum contained in an angle $< \dfrac{2\pi}{n}$ with vertex at the origin. \ Then $T$ has nontrivial invariant subspace.
\end{theorem}
\begin{proof}
	Since $T$ is sub-$n$--normal, $T^{n}$ is hyponormal. \ Then by \cite{duggal2}, $T$ is subscalar. \ The required result follows from \cite{Eschmeier}.
\end{proof}

We conclude this section with a Bram-Embry-type structural result for sub-$n$--normal operators (cf. \cite{Book1}). \ (In the proof of the following result, we will denote by $\mathscr{B}(\mathbb{C})$ and $\mathscr{B}(\mathbb{R}_+)$ the $\sigma$--algebra of Borel subsets of $\mathbb{C}$ and $\mathbb{R}_+$, respectively.) 

\begin{theorem}\label{BE}
	If $T$ is sub-$n$--normal, then there is a positive operator valued measure $F_+$ on some interval $[0,a]$  in $\mathbb{R}$ such that
		$$
		T^{*ni}T^{ni}=\int t^{2ni}\, dF_+(t) \quad \; (i \in \mathbb{Z}_{+}).
		$$
\end{theorem}

\begin{proof}
Since $T$ is sub-$n$--normal, it has an $n$--normal extension $S$ on $\mathcal{K}$. \ Let $E: \mathscr{B}(\mathbb{C}) \rightarrow  B(\mathcal{K})$  be the spectral measure of $S$. \ Let $P$ be the orthogonal projection of $\mathcal{K}$ onto $\mathcal{H}$. \ Then  $F(\Delta):=PE(\Delta)|_{\mathcal{H}}$ \; ($\Delta \in \mathscr{B}(\mathbb{C})$) is a positive operator-valued measure such that $F(\{z\in \mathbb{C}: |z| \geq ||S^{n}||\})=0$. \ It follows that 
$$
T^{*ni}T^{ni}=\int |z|^{2ni}dF(z).
$$
Let $F_{+}: \mathscr{B}(\mathbb{R_{+}}) \rightarrow B(\mathcal{H})$ be defined by 
	$F_{+}(\Delta):=F(h^{-1}(\Delta))$, where $h: \mathbb{C} \rightarrow \mathbb{R_{+}}$ is given by $h(z)=|z|^{2n}$ \; ($z\in \mathbb{C}$). \ Thus  $F_{+}$  is a positive operator valued measure  on the interval $[0,a] \subseteq \mathbb{R}$, where $a:=max\{|z|^{2n}; z \in supp(F)\}$. \ As a result, we obtain   
	$$
	T^{*ni}T^{ni}= \int t^{2ni}dF_+(t),
		$$
as desired.
\end{proof}

\bigskip
\section{The Class of $n$--quasinormal Operators}

\begin{definition}
	An operator $T \in B(H)$ is said to be quasi-$n$--normal if $T$ commutes with $T^{*n}T^{n}$.
\end{definition}

\begin{definition}
	An operator $T \in B(H)$ is said to be $n$--quasinormal if $T^n$ is quasinormal, i.e., $T^n$ commutes with $T^{*n}T^{n}$.
\end{definition}

By a simple calculation, it is evident that if $T$ is  quasi-$n$--normal, then $T$ is $n$--quasinormal. \ The weighted shift $W_\alpha$  with weights $\{a, b, 1,1,1, \ldots \}$, where $0<a,b<1$, is $2$--subnormal (see Example \ref{ee}) but not quasi-$2$--normal (also, not $2$--quasinormal). \ Since $n$--quasinormal operators are $n$--subnormal, we have the following inclusion: \newline
\medskip
$$
\textrm{quasi-$n$--normal } \subseteq \textrm{ $n$--quasinormal } \subseteq \textrm{ $n$--subnormal}.
$$
We now classify the $2$--quasinormal and $3$--quasinormal unilateral weighted shifts. 

\begin{lemma} \label{lem223}
Let $W_\alpha$ be a $2$--quasinormal unilateral weighted shift. \ Then the weight sequence $\alpha$ is periodic with period at most $2$.
\end{lemma}

\begin{proof} \ Without loss of generality, we can assume that $\alpha_0=1$. \ Let $r:=\alpha_1, \quad s:=\alpha_2$. \ Since 
$$
W_\alpha^2 \cong \operatorname{shift}(\alpha_0 \alpha_1, \alpha_2 \alpha_3, \ldots) \oplus \operatorname{shift}(\alpha_1 \alpha_2, \alpha_3 \alpha_4, \ldots)
$$
and $W_\alpha^2$ is quasinormal, we must have $\alpha_0 \alpha_1=\alpha_2 \alpha_3=\ldots$, and $\alpha_1 \alpha_2=\alpha_3 \alpha_4=\ldots$. \ This means: 
$$
\alpha_3=\frac{\alpha_0 \alpha_1}{\alpha_2}, \quad \textrm{ and } \quad \alpha_4=\frac{\alpha_1 \alpha_2}{\alpha_3}, \ldots.
$$
We then have:
$$
W_\alpha=\left(1, r, s, \frac{r}{s}, s^2, \frac{r}{s^2}, s^3, \frac{r}{s^3}, \ldots\right)
$$
Since $W_\alpha$ is bounded, we need $\left\{s^n\right\}$ bounded, i.e., $s \leqslant 1$; on the other hand, we also need $\left\{\frac{1}{s^n}\right\}$ bounded, so $s \geqslant 1$. \ It follows that $s=1$. \ Thus, $W_\alpha=(1, r, 1, r, 1, r, \ldots)$, and $\alpha$ is periodic with period at most $2$.
\end{proof}

\begin{lemma} \label{lem224}
Let $W_\alpha$ be a $3$--quasinormal unilateral weighted shift. \ Then the weight sequence $\alpha$ is periodic with period at most $3$.
\end{lemma}

\begin{proof} \ Without loss of generality, we can assume that $\alpha_0=1$. \ Let $r:=\alpha_1, s:=\alpha_2, \quad t:=\alpha_3, u:=\alpha_4$. \ Since 
$$
W_\alpha^3 \cong \operatorname{shift}(\alpha_0 \alpha_1 \alpha_2, \ldots) \oplus \operatorname{shift}(\alpha_1 \alpha_2 \alpha_3, \ldots) \oplus \operatorname{shift}(\alpha_2 \alpha_3 \alpha_4, \ldots)
$$
and $W_\alpha^3$ is quasinormal, we must have 
$$
\alpha_0 \alpha_1 \alpha_2=\alpha_3 \alpha_4 \alpha_5=\ldots,
$$
$$
\alpha_1 \alpha_2 \alpha_3=\alpha_4 \alpha_5 \alpha_6=\ldots,
$$
and
$$
\alpha_2 \alpha_3 \alpha_4=\alpha_5 \alpha_6 \alpha_7=\ldots.
$$
This means: 
$$
\alpha_5=\frac{r s}{t u}, \quad \alpha_6=\frac{r s t}{u \cdot \alpha_5}, \ldots.
$$
We then have
$$
W_\alpha=\left(1, r, s, t, u, \frac{r s}{t u}, t^2, \frac{u^2}{r}, \frac{r^2 s}{t^2 u^2}, t^3, \frac{u^3}{r^2}, \frac{r^3 s}{t^3 u^3}, \ldots\right).
$$
As before, we need:
$$
\left\{\begin{array} { l } 
{ t \leqslant 1 } \\
{ u \leq r } \\
{ r \leq t u }
\end{array} \Rightarrow \left\{\begin{array}{l}
\frac{r}{u} \leq t \leq 1 \\
\frac{u}{r} \leqslant 1
\end{array} \Rightarrow u=r \Rightarrow t \geqslant 1 .\right.\right.
$$
We conclude that $t=1$ and $u=r_1$. \ It follows that 
$$
W_\alpha=(1, r, s, 1, r, s, 1, r, s, \ldots),
$$
so that $\alpha$ is periodic with period at most $3$.
\end{proof}

Using the technique in Lemmas \ref{lem223} and \ref{lem224}, and with the aid of the software tool {\it Mathematica}, we establish the following result.

\begin{theorem}
Let $W_\alpha$ be an $n$--quasinormal unilateral weighted shift. \ Then the weight sequence $\alpha$ is periodic with period at most $n$.
\end{theorem}

 R.E. Curto, S.H. Lee and J. Yoon
asked the following question. \ Let $T$ be a subnormal operator,
and assume that $T^2$ is quasinormal. \ Does it follow that
$T$ is quasinormal? \ In \cite{CU1}, they proved that this holds when $T$ is injective. \ On the other hand, when $T$ is not necessarily injective, an affirmative answer to this question has been given by P. Pietrzycki and J. Stochel in \cite{JS}. \ We now study a sub-$n$--normal version of Pietrzycki and Stochel's result \cite{JS}, using similar arguments. \ First, we need a lemma.

\begin{lemma}\label{l2}
An operator	$T$ is $n$--quasinormal if and only if $T^{*nk}T^{nk}=(T^{*n}T^{n})^k$ for $k=0,1,2...$
\end{lemma}
\begin{proposition}(\cite{Ha,Mu})\label{p1}
	Let $A\in B(\mathcal{H})$ be a positive operator, $T \in B(\mathcal{H})$ be a contraction
	and $f : [0, \infty) \rightarrow \mathbb{R}$ be a continuous operator monotone function such that $f(0) \geq 0$.
	Then
	$$T^{*}f(A)T \leq f(T^{*}AT).$$
	
	Moreover, if $f$ is not an affine function and $T$ is an orthogonal projection such that
	$T \neq I$, then $T^{*}f(A)T = f(T^{*}AT)$  if and only if $TA = AT$ and $f(0) = 0$.
\end{proposition}
\begin{proposition}(\cite{JS})\label{p2}
	If p is a positive real number, then the commutants of a positive operator
	and of its $p$--th power coincide.
\end{proposition}
\begin{theorem}
	Let $T$ be a sub-$n$--normal operator on a Hilbert space $\mathcal{H}$. \ If $T^{m}$ is  $n$--quasinormal for an integer $m>1$, then $T$ is  $n$--quasinormal.
\end{theorem}
\begin{proof}
	Let $S \in B(\mathcal{K})$ be a $n$--normal extension of $T$. \ We write $$S =	\begin{pmatrix}
		T & U  \\
		0 & V \\
	\end{pmatrix}$$ on $\mathcal{K}=\mathcal{H} \oplus \mathcal{H}^{\perp}$.
	Now $$P(S^{*n}S^{n})^{i}P=P(S^{*ni}S^{ni})P=	\begin{pmatrix}
		(T^{*ni}T^{ni}) & 0 \\
		0 & 0\\
	\end{pmatrix}$$
	
	Since  $T^{m}$ is  $n$ -quasinormal, $T^{*m ni}T^{mi}=(T^{*mn}T^{mn})^{i}$ by Lemma \ref{l2}.
	
	$$P(S^{*n}S^{n})^{k}P=\begin{pmatrix}
		(T^{*nk}T^{nk}) & 0 \\
		0 & 0\\
	\end{pmatrix}=\begin{pmatrix}
		(T^{*nki}T^{nki})^{\frac{1}{i}} & 0 \\
		0 & 0\\
	\end{pmatrix}=(P(S^{*n}S^{n})^{ki}P)^{\frac{1}{i}}$$
	Let $f : [0, \infty) \rightarrow \mathbb{R}$ be the function given by $f(x) = x^{\frac{1}{k}}$
	for $x \in [0, \infty)$. \ It follows by L\"owner-Heinz inequality
	that $f$ is an operator monotone function. \ Using the Stone-von
	Neumann functional calculus, we get
	$$Pf(S^{*n}S^{n})^{ki})P =f( P(S^{*n}S^{n})^{ki}P)$$
	Then by Proposition \ref{p1} P commutes with $(S^{*n}S^{n})^{ki}$. \ Then by proposition \ref{p2}, P commutes with $(S^{*n}S^{n}))$. \ Thus,
	$$\begin{pmatrix}
		(T^{*nk}T^{nk}) & 0 \\
		0 & 0\\
	\end{pmatrix}=P(S^{*n}S^{n})^{k}P=(PS^{*n}S^{n}P)^{k}=\begin{pmatrix}
		(T^{*n}T^{n})^{k} & 0 \\
		0 & 0\\
	\end{pmatrix}$$
	which implies $T$ is $n$--quasinormal.
\end{proof}

\bigskip
\noindent \textit{{\bf Acknowledgments}}. \ The first-named author was partially supported by U.S. NSF grant DMS-2247167. \ The second-named author was supported in part by the Mathematical Research Impact Centric Support, MATRICS (MTR/2021/000373) by SERB,
Department of Science and Technology (DST), Government of India.

\bigskip
\noindent {\bf Declarations}

\medskip

\noindent \underline{Conflict of Interest declaration}: \ The submitted work is original. \ It has not been
published elsewhere in any form or language (partially or in full), and it is
not under simultaneous consideration or in press by another journal. \ The
authors have no competing interests to declare that are relevant to the content
of this article.

\medskip
\noindent \underline{Data availability}: \ The manuscript has no associated data.

\bigskip

	\address{R.E. Curto\endgraf
		Department of Mathematics\endgraf The University of Iowa\endgraf Iowa City, Iowa, 52242, USA}
		
	\email{\textcolor[rgb]{0.00,0.00,0.84}{ raul-curto@uiowa.edu}}
	
	\medskip
	\address{T. Prasad\endgraf
		Department of Mathematics, \endgraf
		University of Calicut,\endgraf
		Kerala-673635, 
		India.}

\email{\textcolor[rgb]{0.00,0.00,0.84}{ prasadvalapil@gmail.com}}

\end{document}